\documentclass[11pt,a4paper]{elsart}
\journal{}

\renewcommand{\theequation}{\thesection.\arabic{equation}}

\usepackage[centertags]{amsmath}
\usepackage{amsfonts}
\usepackage{amssymb}
\usepackage{newlfont}
\usepackage{color}
\usepackage{tabularx}
\usepackage{graphicx}

\newtheorem{remark}{\sc \bf Remark}[section]

\newtheorem{theorem}{\sc \bf Theorem}[section]
\newtheorem{example}{\sc \bf Example}[section]

\newenvironment{proof}{Proof:}{\hfill$\square$}

\begin{document}
\begin{frontmatter}
\title{Imposing various boundary conditions on positive definite kernels}
\author[IKIU]{Babak Azarnavid\corauthref{cor}},
\author[PUT]{Mohammad Nabati},
\author[GUT]{Mahdi Emamjome},
\author[SBU]{Kourosh Parand}
\corauth[cor]{Corresponding author: babakazarnavid@yahoo.com}

\address[IKIU]{Department of Applied Mathematics, Imam Khomeini International
University, Qazvin 34149-16818, Iran}
\address[PUT]{Department of Basic Sciences, Abadan Faculty of Petroleum
engineering, Petroleum University of Technology, Abadan, Iran}
\address[GUT]{Department of Mathematics, Golpayegan University of
Technology, Golpayegan, P.O.Box 8771765651, Iran}
\address[SBU]{Department of Computer Sciences, Shahid Beheshti University, G.C. Tehran 19697-64166, Iran}

\begin{abstract}

This work is motivated by the frequent occurrence of boundary value
problems with various boundary conditions in the modeling of some
problems in engineering and physical science. Here we propose a new
technique to force the positive definite kernels such as some radial
basis functions to satisfy the boundary conditions exactly. It can
improve the applications of existing methods based on positive
definite kernels and radial basis functions especially the
pseudospectral radial basis function method for handling the
differential equations with more complicated boundary conditions.
The proposed method is applied to a singularly perturbed
steady-state convection-diffusion problem, two and three dimensional
Poisson's equations with various boundary conditions.

\end{abstract}

\end{frontmatter}

\section{Introduction}
Setting up and solving models described by the Poisson's equation is
one of the cornerstones of electrostatics. The Poisson's equations
combined with the various type of boundary conditions are of
importance for a wide field of applications in electrostatics,
mechanical engineering, theoretical and computational physics and
theoretical chemistry. This work is motivated by the frequent
occurrence of boundary value problems with various boundary
conditions in the modeling of some problems in engineering and
physical science. In recent years, several kernel based algorithms
have been proposed for solving boundary value problems
\cite{roohani,a1,a2,parand,kansa,shivanian1,shivanian2,roohani2,p2,p3,babak,babak1_17,babak3,babak2,babak5,br2,br3,arq6,arq3,arq4,arq5,akgul_r1,akgul_r7,parand-amani,fassh1}.
Fasshauer \cite{fassh1} has shown that many of the standard
algorithms and strategies used for solving ordinary and partial
differential equations with polynomial pseudospectral methods can be
easily adapted for the use with radial basis functions.
pseudospectral radial basis function (RBF--PS) method has already
been proven successful in numerical solution of various type of
differential equation
\cite{fassh1,sarra,fassh2,uddin1,fassh3,fassh4,uddin2,uddin3,forn1}.
This paper presents a new approach to imposing various boundary
conditions on positive definite kernels such as some radial basis
functions and their application in kernel based pseudospectral
method. The various boundary conditions such as Dirichlet, Neumann,
Robin, mixed and multi--point boundary conditions, have been
considered. Here we propose a new technique to force the radial
basis functions to satisfy the boundary conditions exactly, so the
approximate solution also satisfies the boundary conditions exactly.
It can improve the applications of existing methods based on
positive definite kernels and radial basis functions especially the
pseudospectral radial basis function method for handling the
differential equations with more complicated boundary conditions.
Some new kernels are constructed using general kernels in a manner
which satisfies required conditions and we prove that if the
reference kernel is positive definite then the newly constructed
kernel is positive definite, also. Furthermore, we show that the
collocation matrix is nonsingular if some conditions are satisfied.
In \cite{fassh1} RBF--PS method has been applied successfully to
homogeneous Dirichlet boundary conditions. Here we try to handle
many other types of boundary conditions in one, two and three
dimension. The proposed technique can be applied to other kinds of
radial basis functions methods easily. Imposing boundary conditions
is a key issue in meshless methods based on radial basis functions
and can be quite challenging. We shall discuss how to deal with
boundary conditions in radial basis functions methods. The Dirichlet
boundary condition is relatively easy and other type boundary
conditions require more attentions. There are two basic approaches
to deal with boundary conditions for pseudospectral methods,
restrict the method to basis functions that satisfy the boundary
conditions or add some additional equations to enforce the boundary
conditions. An inherent advantage of the proposed technique is its
simplicity and easy programmability. Difficulties in the various
radial basis function method arise in applying the method to a
boundary value problem with more
complicated nonhomogeneous boundary conditions in each dimension such as:\\
Let $u$ be the solution which we are looking for and $\Omega=[a,b]$
is the domain of problem in one direction:

$1.$ the Dirichlet boundary condition
\begin{equation}\label{bc3}
 u(a)=A \hspace{1 cm} and \hspace{1 cm}u(b)=B,
\end{equation}

$2.$ the Neumann boundary condition
\begin{equation}\label{bc3}
 u'(a)=A \hspace{1 cm} and \hspace{1 cm}u'(b)=B,
\end{equation}

$3.$ the mixed boundary condition
\begin{equation}\label{bc3}
 u(a)=A \hspace{1 cm} and \hspace{1 cm}u'(b)=B,
\end{equation}

$4.$ the Robin boundary condition
\begin{equation}\label{bc2}
\begin{array}{ll}
\alpha_{1} u(a)+\beta_{1} u'(a)&=c_{1},\\
\alpha_{2}u(b)+\beta_{2} u'(b)&=c_{1},\\
\end{array}
\end{equation}

$5.$ the Multi--point boundary condition
\begin{equation}\label{bc1}
u(a)=\sum_{j=1}^{J}\alpha_{j}u(\xi_{j})+\psi,
\end{equation}
where
\begin{equation*}\label{bc01}
a<\xi_{1}<\xi_{2}<...<\xi_{J}<b.
\end{equation*}

In fact, the Dirichlet and Neumann boundary conditions are special
cases of the Robin boundary condition. The presented technique is
easy to utilize by existing radial basis function method for
handling more complicated boundary conditions. Several test examples
are presented to demonstrate the accuracy and versatility of the
proposed technique. We apply it to some problems in one, two and
three dimensional with the different type of boundary conditions and
compare the results with the RBF collocation method introduced in
\cite{fassh1} and the best reported results in the literature. The
reported results show that the proposed method is accurate and
significantly more efficient than RBF collocation method and some
other existing radial basis functions method.

%************************%
%            *           %
%            *           %
%            *           %
%            *           %
%************************%

\section{Kernel based pseudospectral method}

In this section, we give a brief review of the pseudospectral method
based on kernels. An important feature of pseudospectral methods is
the fact that one usually is content with obtaining an approximation
to the solution on a discrete set of grid points. In pseudospectral
methods we usually seek an approximate solution of differential
equation in the form
\begin{equation}\label{e:006}
u_{N}(x)=\sum_{j=1}^{N}\lambda_{j}\phi_{j}(x).
\end{equation}
For the grid points $x_{i},i=1,...,N,$ We will use the basis
functions $\phi_{j}(x)=R(x,x_{j}),$ where $R(x,y)$ is a kernel. If
we evaluate the unknown function $u(x)$ at grid points
$x_{i},i=1,...,N,$ then we have,
\begin{equation}\label{e:01}
u_{N}(x_{i})=\sum_{j=1}^{N}\lambda_{j}\phi_{j}(x_{i}),\hspace{1 cm}
i=1,...,N,
\end{equation}
or in matrix notation,
\begin{equation}\label{e:02}
\boldsymbol{u}=A\boldsymbol{\lambda},
\end{equation}
where $\boldsymbol{\lambda}=[\lambda_{1},...,\lambda_{N}]^{T}$ is
the coefficient vector, the evaluation matrix $A$ has the entries
$A_{i,j}=\phi_{j}(x_{i})=R(x_{i},x_{j})$ and
$\boldsymbol{u}=[u_{N}(x_{1},...,u_{N}(x_{N})]^{T}$. Let
$\mathcal{L}$ be a linear operator, we can use the expansion
\eqref{e:006} to compute the $\mathcal{L}u_{N}$ by operating
$\mathcal{L}$ on the basis functions,
\begin{equation}\label{e:8}
\mathcal{L}u_{N}=\sum_{j=1}^{N}\beta_{j}\mathcal{L}\phi_{j}(x),\hspace{1
cm} x\in \mathbb{R}^{d}.
\end{equation}
If we again evaluate at the grid points $x_{i},i=1,...,N,$ then we
get in matrix notation,
\begin{equation}\label{e:9}
\boldsymbol{L}\boldsymbol{u}=A_{L}\boldsymbol{\lambda},
\end{equation}
where $\boldsymbol{u}$ and $\boldsymbol{\lambda}$ are as above and
the matrix $A_{L}$ has entries $\mathcal{L}\phi_{j}(x_{i})$. Then we
can use \eqref{e:02} to solve the coefficient vector
$\boldsymbol{\lambda}=A^{-1}\boldsymbol{u},$ and then \eqref{e:9}
yields,
\begin{equation}\label{e:10}
\boldsymbol{L}\boldsymbol{u}=A_{L}A^{-1}\boldsymbol{u},
\end{equation}
so that the operational matrix $\boldsymbol{L}$ corresponding to
linear operator $L$ is given by,
\begin{equation}\label{omat}
\boldsymbol{L}=A_{L}A^{-1}.
\end{equation}
In order to obtain the differentiation matrix $\boldsymbol{L}$ we
need to ensure invertibility of the evaluation matrix $A$. This
generally depends both on the basis functions chosen as well as the
location of the grid points $x_{i},i=0,...,n$. For positive definite
kernels the invertibility of the evaluation matrix $A$ for any set
of distinct grid points is guaranteed. Suppose we have a linear
differential equation of the form
\begin{equation}\label{e0011}
\mathcal{L}u=f,
\end{equation}
by ignoring boundary conditions. An approximate solution at the grid
points can be obtained by solving the discrete linear system
\begin{equation}\label{e0012}
\textbf{L}\textbf{u}=\textbf{f},
\end{equation}
where $\textbf{u}$ and $\textbf{f}$ contain the value of $u$ and $f$
at grid points and $\textbf{L}$ is the mentioned operational matrix
corresponds to linear differential operator $\mathcal{L}$. Imposing
boundary conditions in radial basis functions methods based on
radial basis functions and can be quite challenging. Here we impose
boundary conditions on basis functions, instead of add some
additional equations in \eqref{e0012} to enforce the boundary
conditions. Many radial basis functions are defined by a constant
called the shape parameter. The choice of shape parameter have a
significant impact on the accuracy of an radial basis function
method. It is clear that selecting optimal shape parameter in the
methods based on the radial basis functions is an open problem. But
authors of \cite{fassh3} proposed an algorithm for choosing an
optimal value of the shape parameter. Here we consider the effect of
different shape parameters on the accuracy of approximations and
compare it with the RBF collocation method.

%***********************%
%                       %
%                       %
%                       %
%                       %
%***********************

\section{Imposing the boundary conditions}

For some nonhomogeneous problems, we can construct a homogenization
function $M$, which satisfies the nonhomogeneous boundary conditions
of problem. Then the nonhomogeneous problem can be reduced to a
homogeneous problem as follows. Let
\begin{equation*}\label{bc_r1}
 \mathcal{L}u(\textbf{x})=f(\textbf{x}),\hspace{.5 cm} \textbf{x}\in \Omega\subset R^{d},\ \mathcal{B}u(\textbf{x})=g(\textbf{x}),\hspace{.5 cm} \textbf{x}\in \partial \Omega
\end{equation*}
where $\partial \Omega$ is the boundary of $\Omega$ and
$\mathcal{L}$ is a differential operator. Then the boundary
conditions can be homogenized using
\begin{equation*}\label{bc0_r2}
u(\textbf{x})=v(\textbf{x})+M(\textbf{x}).
\end{equation*}
After homogenization of the boundary conditions, the nonhomogeneous
problem can be convert in the following form
\begin{equation*}\label{bc_r1}
 \mathcal{L}v(\textbf{x})=F(\textbf{x}),\hspace{.5 cm} \textbf{x}\in \Omega\subset R^{d},\ \mathcal{B}v(\textbf{x})=0,\hspace{.5 cm} \textbf{x}\in \partial \Omega
\end{equation*}
where $F(\textbf{x})=f(\textbf{x})-\mathcal{L}M(\textbf{x})$. For
example in the following we construct the homogenization function
$M$ for two dimensional Dirichlet boundary conditions. Let we have
the following boundary conditions
\begin{equation}\label{1_h}
\begin{array}{llll}
u(a,y)&=g_{1}(y),&&u(b,y)=g_{2}(y),\\
u(x,c)&=h_{1}(x), &&u(x,d)=h_{2}(x),\\
\end{array}
\end{equation}
then the following function $M$ satisfies the nonhomogeneous
conditions and we can homogenize the boundary conditions using
$u=v+M$ which $v$ satisfies the homogenous conditions
\begin{equation}\label{2_h}
\begin{array}{ll}
M_{1}(x,y)=\frac{x-b}{a-b}g_{1}(y)+\frac{x-a}{b-a}g_{2}(y),\\
M(x,y)=M1(x,y)+\frac{y-d}{c-d}(h_{1}(x)-M_{1}(x,c))+\frac{y-c}{d-c}(h_{2}(x)-M_{1}(x,d)),\\
\end{array}
\end{equation}
we can easily see that $M$ satisfies the nonhomogeneous boundary
conditions \eqref{1_h}. For other type of boundary conditions the
homogenization function $M$ can be constructed in a similar way. In
the proposed method, firstly, the nonhomogeneous problem is reduced
to a homogeneous one and then the homogenous conditions are imposed
on kernel function.

Let
\begin{equation}\label{3_h}
L_{1}v=0,\hspace{.5 cm}L_{2}v=0,
\end{equation}
be the homogenous conditions in $x$ direction. In the next theorem
the kernel function is constructed using the reference kernel
$R(x,y)$ such that satisfies \eqref{3_h}.
\begin{equation}\label{bc02_6}
R_{1}(x,y)=R(x,y)-\frac{L_{1,x}R(x,y)L_{1,y}R(x,y)}{L_{1,x}L_{1,y}
R(x,y)},
\end{equation}
and
\begin{equation}\label{bc02_7}
R_{2}(x,y)=R_{1}(x,y)-\frac{L_{2,x}R_{1}(x,y)L_{2,y}R_{1}(x,y)}{L_{2,x}L_{2,y}
R_{1}(x,y)}.
\end{equation}
where the subscript $x$ on the operator $L$ indicates that the
operator $L$ applies to the function of $x$.

\begin{theorem}\label{te2}
If $L_{1,x}L_{1,y} R(x,y)\neq 0$ and $L_{2,x}L_{2,y} R_{1}(x,y)\neq
0$, then $R_{2}(x,y)$ given by \eqref{bc02_7} satisfies the boundary
conditions \eqref{3_h} exactly.
\end{theorem}

\begin{proof}
By applying the operator $L_{1,x}$ to $R_{1}(x,y)$ we have
\begin{equation*}\label{bc02_8}
\begin{array}{ll}
L_{1,x}R_{1}(x,y)&=L_{1,x}R(x,y)-\frac{L_{1,x}R(x,y)L_{1,x}L_{1,y}R(x,y)}{L_{1,x}L_{1,y}
R(x,y)}\\&=L_{1,x}R(x,y)-L_{1,x}R(x,y)=0,\\
\end{array}
\end{equation*}
and
\begin{equation*}\label{bc02_9}
\begin{array}{ll}
L_{1,x}L_{2,y}R_{1}(x,y)&=L_{1,x}L_{2,y}R(x,y)-\frac{L_{2,y}L_{1,x}R(x,y)L_{1,x}L_{1,y}R(x,y)}{L_{1,x}L_{1,y}
R(x,y)}\\
&=L_{1,x}L_{2,y}R(x,y)-L_{2,y}L_{1,x}R(x,y)=0,\\
\end{array}
\end{equation*}
then
\begin{equation*}\label{bc02_10}
L_{1,x}R_{2}(x,y)=L_{1,x}R_{1}(x,y)-\frac{L_{2,x}R_{1}(x,y)L_{1,x}L_{2,y}R_{1}(x,y)}{L_{2,x}L_{2,y}.
R_{1}(x,y)}=0,
\end{equation*}
By applying the operator $L_{2,x}$ to $R_{2}(x,y)$ we have
\begin{equation*}\label{bc02_9}
\begin{array}{ll}
L_{2,x}R_{2}(x,y)&=L_{2,x}R_{1}(x,y)-\frac{L_{2,x}R_{1}(x,y)L_{2,x}L_{2,y}R_{1}(x,y)}{L_{2,x}L_{2,y}
R_{1}(x,y)}\\&=L_{2,x}R_{1}(x,y)-L_{2,x}R_{1}(x,y)=0.\\
\end{array}
\end{equation*}
\end{proof}
\begin{theorem}\label{the1}
Let real valued symmetric positive definite kernel $R(x,y)$ be the
reproducing kernel of Hilbert space $H$ defined on a region $\Omega$
and $L:H\rightarrow R$ be a continuous linear functional and
$L_{x}L_{y}R(x,y)\neq 0$. Then
$R_{0}(x,y)=R(x,y)-\frac{L_{x}R(x,y)L_{y}R(x,y)}{L_{x}L_{y}
R(x,y)},$ is also a real valued positive definite kernel.
\end{theorem}
\begin{proof}
Let $H_{0}=\{u\in H: Lu=0\}$. Then $H_{0}$ is a closed subspace of
$H$ and it is a reproducing kernel Hilbert space. Now we prove that
$R_{0}(x,y)$ is the symmetric reproducing kernel of $H_{0}$ so it is
a symmetric positive definite kernel. based on the Riesz'
representation theorem there exists $g\in H$ such that for all $u\in
H$ we have $Lu=(g,u)_{H}$. Then
\begin{equation*}\label{eq00}
L_{y}R(x,y)=(g,R(x,.))_{H}=g(x)\in H,
\end{equation*}
where the lower index shows the variable that the functional acts
on. So for any $\bar{x}\in \Omega$ we have
\begin{equation}\label{eq01}
R_{0}(x,\bar{x})=R(x,\bar{x})-\frac{L_{x}R(x,\bar{x})L_{y}R(x,y)}{L_{x}L_{y}R(x,y)}=R(x,\bar{x})-\alpha
g(x)\in H,
\end{equation}
for some $\alpha\in R$. Also we have
\begin{equation}\label{eq02}
L_{x}R_{0}(x,y)=L_{x}R(x,y)-\frac{L_{x}R(x,y)L_{x}L_{y}R(x,y)}{L_{x}L_{y}R(x,y)}=0.
\end{equation}
From \eqref{eq01} and \eqref{eq02} we can see that for any
$\bar{x}\in \Omega$, $R_{0}(x,\bar{x})\in H_{0}$. For any $u\in
H_{0}$ we have
\begin{equation}\label{eq03}
\begin{array}{ll}
(R_{0}(x,.),u)_{H}&=(R(x,.),u)_{H}-\frac{L_{x}(R(x,.),u)_{H}L_{y}R(x,y)}{L_{x}L_{y}R(x,y)}\\
&=u(x)-\frac{L_{x}u(x)L_{y}R(x,y)}{L_{x}L_{y}R(x,y)}=u(x),\\
\end{array}
\end{equation}
which shows the reproducing property of $R_{0}(x,y)$ in $H_{0}$. It
is easy to see that $R_{0}(x,y)$ is symmetric reproducing kernel of
$H_{0}$ so it is a symmetric positive definite kernel.
\end{proof}

A real valued positive definite kernel $R(x,y)$ leads to a real
Hilbert space of real valued functions named native space
\cite{wend}. Based on previous theorem the well posedness, stability
estimates and other features of symmetric positive definite kernel
based methods, proved in \cite{wend,aron}, for new constructed
kernels are all Still hold.

\begin{theorem}\label{the1}
If $R(x,y)$ be the reproducing kernel of reproducing kernel Hilbert
space $H$ defined on a region $\Omega\subset R^{d}$ and
$\mathfrak{L}:H\rightarrow R$ be such that $\mathfrak{L}u=0$ for
$u\in H$ result that $u=0$. Then the operator matrix $A_{L}$ is
nonsingular.
\end{theorem}
\begin{proof}
Let the matrix $A_{L}$ has entries
$\mathfrak{L}\phi_{j}(x_{i})=\mathfrak{L}R(x_{i},x_{j})$ and $c\in
R^{n}$ be an arbitrary vector then
\begin{equation*}\label{eq_t_0}
\begin{array}{ll}
c_{1}\mathfrak{L}\phi_{1}(x)+c_{2}\mathfrak{L}\phi_{2}(x)+...+c_{n}\mathfrak{L}\phi_{n}(x)=\\
c_{1}\mathfrak{L}R(x,x_{1})+c_{2}\mathfrak{L}R(x,x_{2})+...+c_{n}\mathfrak{L}R(x,x_{n})=0,
\end{array}
\end{equation*}
then
\begin{equation*}\label{eq_t_1}
\mathfrak{L}(c_{1}R(x,x_{1})+c_{2}R(x,x_{2})+...+c_{n}R(x,x_{n}))=0,
\end{equation*}
then we have
\begin{equation*}\label{eq_t_2}
c_{1}R(x,x_{1})+c_{2}R(x,x_{2})+...+c_{n}R(x,x_{n})=0,
\end{equation*}
from positive definiteness of kernel $R(x,y)$ it is easy to see that
$c=0$ and $\phi_{i}(x),i=1,...,n$ are linearly independent and so
$A_{L}$ is nonsingular.
\end{proof}

For solving multi-dimensional problems we are using the product of
positive definite kernels as kernels in multi-dimensional domain and
it is the reproducing kernel of a reproducing kernel Hilbert space
and is strictly positive definite kernel.
\begin{theorem} \label{th002} \cite{aron}
Let $H_{1}$ and $H_{2}$ be reproducing kernel spaces with
reproducing kernels $R_{1}$ and $R_{2}$. The direct product
$\overline{H}=H_{1}\bigotimes H_{2}$ is a reproducing kernel Hilbert
space and possesses the reproducing kernel
$\overline{R}(x_{1},x_{2},y_{1},y_{2})=R_{1}(x_{1},y_{1})R_{2}(x_{2},y_{2})$.
\end{theorem}

\begin{remark}
For multidimensional problems we can use any radial or other
positive definite kernel for each direction as reference kernel.
\end{remark}

%***********************%
%                       %
%                       %
%                       %
%                       %
%***********************

\section{Numerical examples}
In this section, some numerical examples are considered, to
illustrate the performance and computation efficiency of new
technique. We consider a one dimensional singularly perturbed
steady-state convection dominated convection–-diffusion problem with
Robin boundary conditions as first example. The proposed method is
used to approximate the solutions of the two and three dimensional
Poisson's equations with various boundary conditions, which are of
importance for a wide field of applications in computational physics
and theoretical chemistry. The numerical results are compared with
the RBF collocation method introduced in \cite{fassh1} and the best
results reported in the literature
\cite{ansari,yun,fassh_mixed,schaback2,volkov,tsai}. For all
examples we use the Gaussian radial basis function.

%%%%%%%%%%%%%%%%%%%ex1
\begin{example}\label{ex1}
Consider the following singularly perturbed convection diffusion
problem \cite{ansari},
\begin{equation*}\label{ex1:01}
\epsilon u''+\frac{1}{1+x}u'=x+1,
\end{equation*}
with the Robin boundary conditions
\begin{equation*}\label{ex1:02}
u(0)-\epsilon u'(0)=1,\hspace{.5 cm} u(1)+u'(1)=1.
\end{equation*}
The exact solution of problem is given by
\begin{equation*}\label{ex1:03}
u=\frac{(x+1)^{3}}{3(2\epsilon+1)}+D\left(\frac{(x+1)^{1-\frac{1}{\epsilon}}}{\epsilon-1}-(\frac{2^{^{1-\frac{1}{\epsilon}}}}{\epsilon-1}+\frac{2^{-\frac{1}{\epsilon}}}{\epsilon})\right)+(1+\frac{20}{3(2\epsilon+1)}),
\end{equation*}
where
\begin{equation*}\label{ex1:04}
D=\frac{(19+3\epsilon)/(3(2\epsilon+1))}{((1-2^{1-\frac{1}{\epsilon}})/(\epsilon-1)-2^{-1/\epsilon}/\epsilon)-1}.
\end{equation*}
\begin{table}
\centering \scriptsize{\begin{tabular}{ |c|c|c|c|c| }
  \hline
 N&$\epsilon$& \cite{ansari}  & RBF collocation & Presented method \\
 \hline
 32&$2^{-1}$&7.93e-2 &2.151530648e-17& 1.677759019e-18\\
 \hline
 64&$2^{-1}$&4.02e-2 &2.896067662e-36& 2.217079325e-37\\
 \hline
  128&$2^{-1}$&2.02e-2 &2.141728769e-74& 1.623426611e-75\\
 \hline
  32&$2^{-5}$&6.62e-1 &2.909789773e-4& 1.580306190e-6\\
 \hline
  64&$2^{-5}$&4.04e-1 &2.179862305e-16& 1.182127709e-18\\
 \hline
  128&$2^{-5}$&2.38e-1 &8.050946529e-47& 3.898941782e-49\\
 \hline
  128&$2^{-10}$&2.68e-1 &6.224300576& 3.310984775e-1\\
 \hline
  256&$2^{-10}$&1.54e-1 &1.657779099& 9.155282792e-4\\
 \hline
\end{tabular}}
\caption{\scriptsize{Maximum absolute errors, comparison of results
for Example \ref{ex1}.}}\label{tab_ex1_1}
\end{table}
\begin{figure}
\centering
\includegraphics[scale=.8]{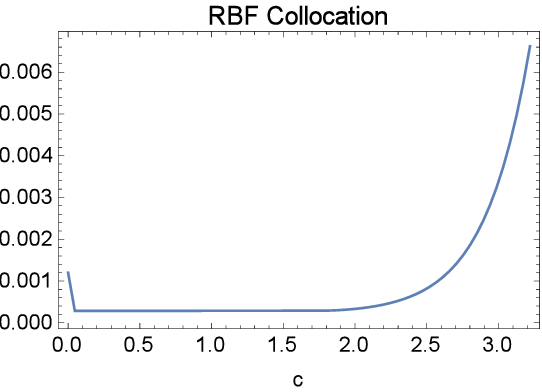}
\includegraphics[scale=.8]{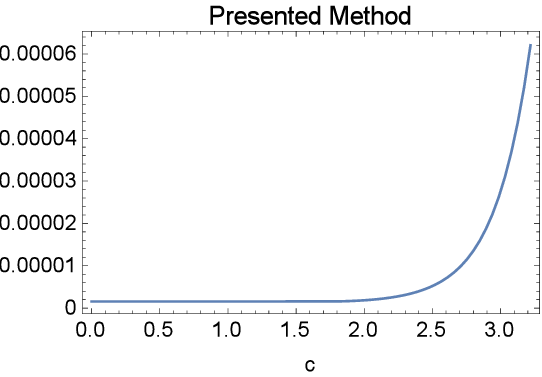}
\caption{\scriptsize{ Graphs of maximum absolute error versus shape
parameter with Gaussian RBF, $N=32$ and $\epsilon=\frac{1}{2^{5}}$,
for Example \ref{ex1}.}}\label{fig_ex1_1}
\end{figure}
For this example, the maximum absolute errors are presented in
Table~\ref{tab_ex1_1} for various values of $N$ and $\epsilon$ and
they are compared with the best reported results in \cite{ansari}
and RBF collocation method. The Gaussian RBF with $c=\frac{18}{100}$
is used for presented method and RBF collocation method. Graphs of
maximum absolute error versus shape parameter with Gaussian RBF,
$N=32$ and $\epsilon=\frac{1}{2^{5}}$ are given in
Figure~\ref{fig_ex1_1}. The reported results show that more accurate
approximate solutions can be obtained using more mesh points. The
numerical simulations show that the presented method is robust and
accurate and remains stable as shape parameter gets smaller in
contrast with the existing radial basis functions methods.
\end{example}
%%%%%%%%%%%%%%%%%%%ex2
\begin{example}\label{ex2}
Consider the Poisson's equation,
\begin{equation*}\label{ex2:01}
-\Delta u=-\left(\frac{\partial^{2}u}{\partial
x^{2}}+\frac{\partial^{2}u}{\partial y^{2}}\right)=f(x,y),\hspace{.5
cm} (x,y)\in[0,1]\times[0,1]
\end{equation*}
with the Robin boundary conditions
\begin{equation*}\label{ex2:02}
\begin{array}{ll}
u|_{x=0}=0,\hspace{.5 cm}u|_{y=0}=0,\\
u_{v}|{x=1}=g_{N},\\
(u_{v}+\alpha u)|_{y=1}=g_{R},\\
\end{array}
\end{equation*}
where $\alpha>0$,(e.g.,$\alpha=2$) and $v$ is the outward normal
vector to the boundary. The functions $f, g_{N}$ and $g_{R}$ are
given such that the exact solution is \cite{yun},
\begin{equation*}\label{ex2:03}
u=\sin(\frac{\pi x}{6})\sin(\frac{7\pi x}{4})\sin(\frac{3\pi
y}{4})\sin(\frac{5\pi y}{4}).
\end{equation*}

\begin{table}
\centering \scriptsize{\begin{tabular}{ |c|c|c|c|c| }
  \hline
 N&$7\times 7$& $9\times 9$  & $11\times 11$ & $13\times 13$ \\
 \hline
 \cite{yun}&9.30e-3&5.92e-5 &4.32e-6& 1.10e-6\\
 \hline
 RBF collocation&3.31818e-3&3.03747e-4 &6.31077e-6& 1.06431e-7\\
 \hline
 Presented method&2.64223e-4&1.42617e-5 &2.11003e-7& 1.0773e-8\\
 \hline
\end{tabular}}
\caption{\scriptsize{Maximum absolute errors, comparison of results
for Example \ref{ex2}.}}\label{tab_ex2_1}
\end{table}
\begin{figure}
\centering
\includegraphics[scale=.6]{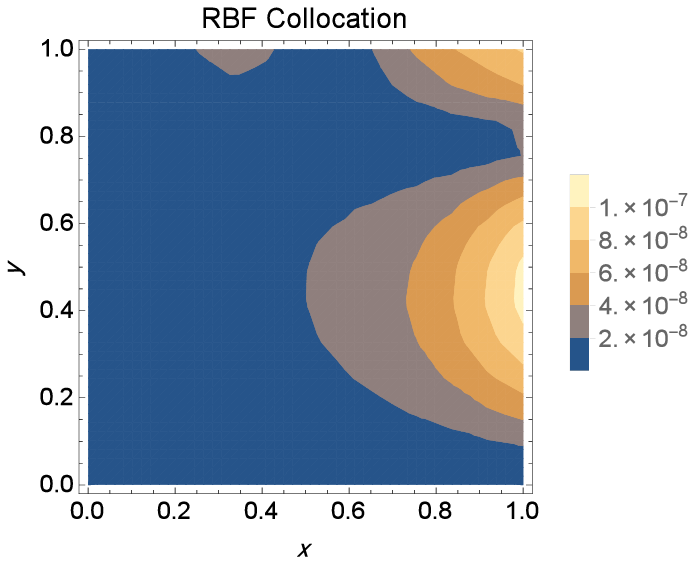}
\includegraphics[scale=.6]{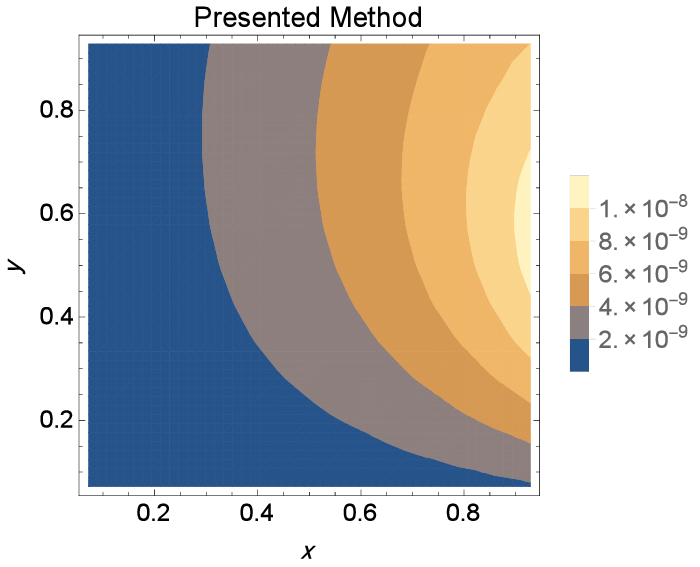}
\caption{\scriptsize{Distribution of the absolute error with
Gaussian RBF, $N=13\times13$ and $c=1$, for Example
\ref{ex2}.}}\label{fig_ex2_1}
\end{figure}
\begin{figure}
\centering
\includegraphics[scale=.8]{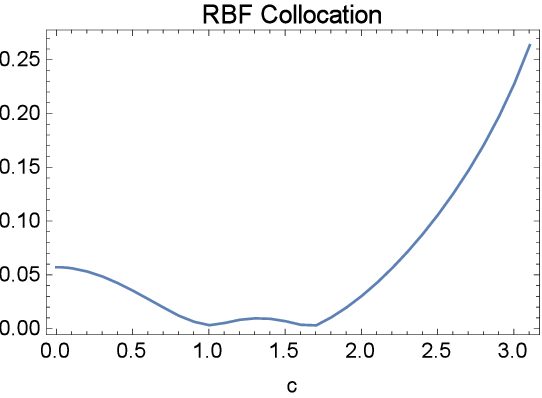}
\includegraphics[scale=.8]{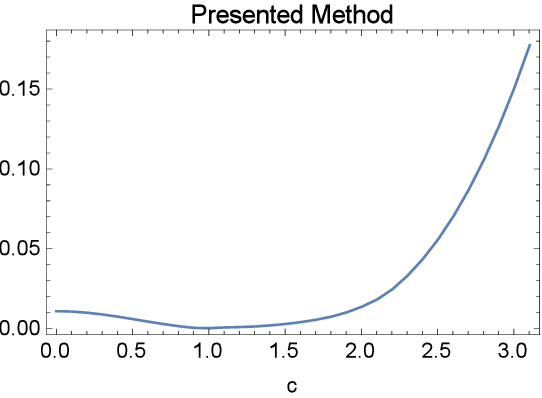}
\caption{\scriptsize{ Graphs of maximum absolute error versus shape
parameter with Gaussian RBF and $N=7\times7$, for Example
\ref{ex2}.}}\label{fig_ex2_2}
\end{figure}
The maximum absolute errors are presented in Table~\ref{tab_ex2_1}
for various values of $N$ and they are compared with the reported
results in \cite{yun} and RBF collocation method. The Gaussian RBF
with $c=1$ is used for presented method and RBF collocation method.
Figure~\ref{fig_ex2_1} shows the distribution of the absolute error
of presented method and RBF collocation method with Gaussian RBF,
$N=13\times13$ and $c=1$. The reported results show that more
accurate approximate solutions can be obtained using more mesh
points. Comparison of numerical results show that the presented
method has the exponential convergence rates and is more accurate
than RBF collocation method and combination of RBF collocation and
Ritz--Galerkin method \cite{yun}. Graphs of maximum absolute error
versus shape parameter with Gaussian RBF, $N=7\times7$ are given in
Figure~\ref{fig_ex2_2} which show that remain stable as shape
parameter gets smaller in contrast with the existing radial basis
functions methods.
\end{example}
%%%%%%%%%%%%%%%%%%%ex3
\begin{example}\label{ex3}
Consider the Poisson's equation \cite{fassh_mixed},
\begin{equation*}\label{ex3:01}
\Delta u=y(1-y)\sin^{3}x,\hspace{.5 cm} (x,y)\in[0,\pi]\times[0,1]
\end{equation*}
with the homogeneous Dirichlet boundary conditions
\begin{equation*}\label{ex3:02}
u(x,0)=u(x,1)=u(0,y)=u(\pi,y)=0.
\end{equation*}
The exact solution is given by
\begin{equation*}\label{ex3:03}
\begin{array}{ll}
u(x,y)&=\frac{3 e^{-y}\left(-2 e -2
e^{2y}+e^{y}(1+e)(2+(-1+y)y)\right)\sin(x)}{4(1+e)}\\
&+\frac{3 e^{-3y}\left(2 e^{3} +2
e^{6y}-e^{3y}(1+e^{3})(2+9(-1+y)y)\right)\sin(3x)}{324(1+e^{3})}.\\
\end{array}
\end{equation*}
\begin{table}
\centering \scriptsize{\begin{tabular}{ |c|c|c|c|c|c| }
  \hline
 N&$\rho_{K}$\cite{fassh_mixed}& $\rho_{H}$\cite{fassh_mixed} & $\rho_{RBFC}$ & $\rho_{1}$&$\rho_{2}$\\
 \hline
 $8\times4$&1.103747e-2&1.062891e-2 &7.4357e-2& 2.1191e-3&2.84849e-3\\
 \hline
 $10\times6$&2.739293e-3&3.451799e-3 &1.58122e-3& 3.60844e-5&3.10566e-4\\
 \hline
$16\times8$&2.707006e-4&2.082886e-4 &1.92361e-5& 5.17671e-7&1.59593e-7\\
 \hline
 $20\times12$&3.894511e-5&1.273363e-5 &3.60382e-9& 8.72329e-11&6.0899e-11\\
 \hline
\end{tabular}}
\caption{\scriptsize{Relative errors, comparison of results for
Example \ref{ex3}.}}\label{tab_ex3_1}
\end{table}
\begin{figure}
\centering
\includegraphics[scale=.6]{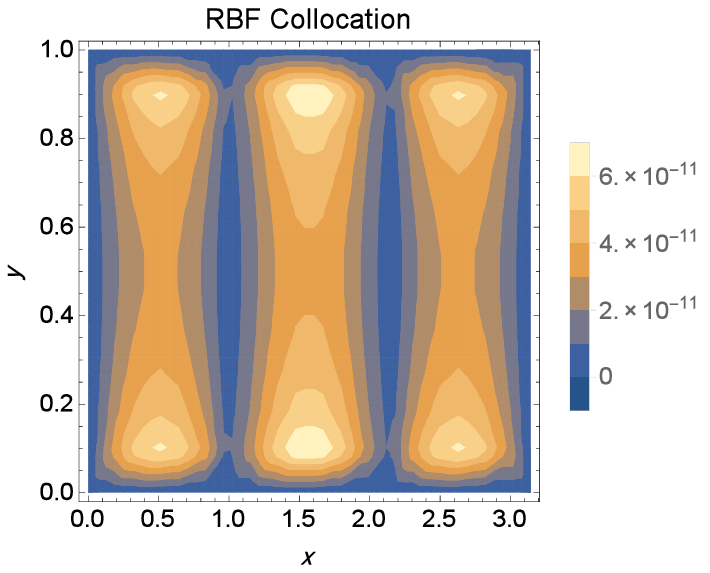}
\includegraphics[scale=.6]{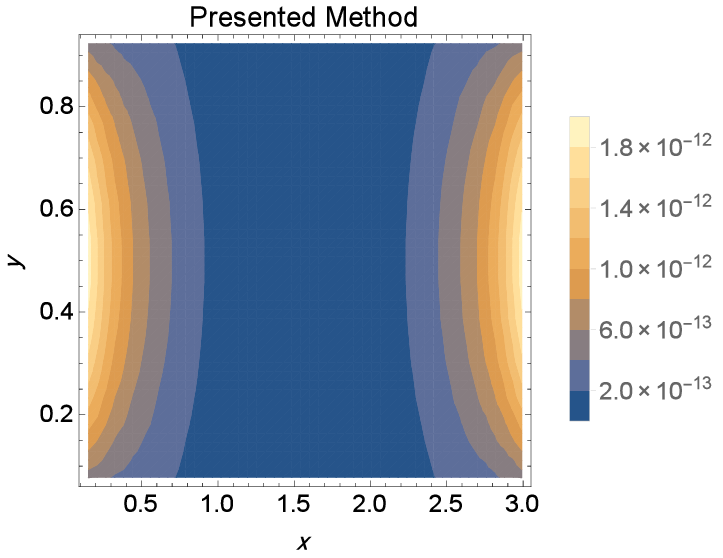}
\caption{\scriptsize{Distribution of the absolute error with
Gaussian RBF, $N=20\times12$ and $c=0.3041$ for RBF collocation and
$c=0.01$ for Presented method, for Example
\ref{ex3}.}}\label{fig_ex3_1}
\end{figure}
\begin{figure}
\centering
\includegraphics[scale=.8]{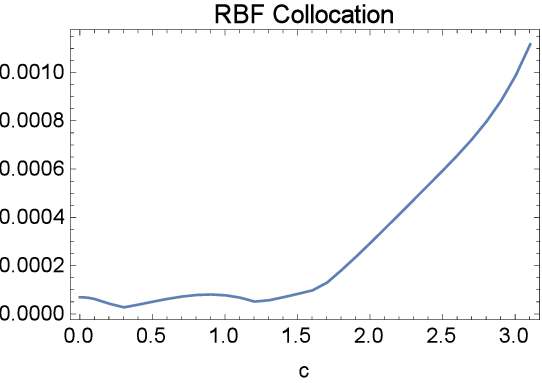}
\includegraphics[scale=.8]{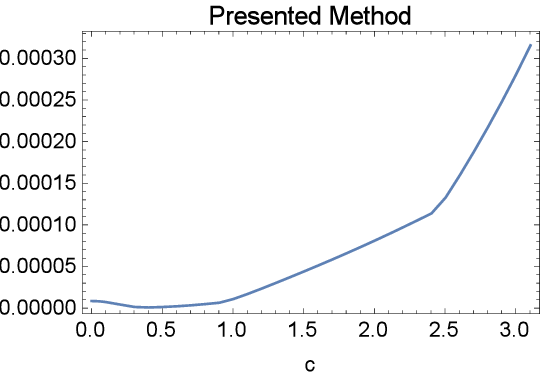}
\caption{\scriptsize{ Graphs of maximum absolute error versus shape
parameter with Gaussian RBF, $N=10\times6$, for Example
\ref{ex3}.}}\label{fig_ex3_2}
\end{figure}
The relative errors are presented in Table~\ref{tab_ex3_1} for
various values of $N$ and they are compared with the best reported
results in \cite{fassh_mixed} contain Kansa's and Hermit based RBF
method and RBF collocation method. $\rho_{RBFC}$ is relative error
of RBF collocation method with Gaussian RBF with $c=0.3041$ and
$\rho_{1}$ and $\rho_{2}$ are relative errors of presented method
with Gaussian RBF with $c=0.4041$ and $c=0.01$, respectively.
$\rho_{K}$ and $\rho_{H}$ are reported relative errors in
\cite{fassh_mixed} with optimal shape parameters for Kansa's and
Hermit based RBF method, respectively. Figure~\ref{fig_ex3_1} shows
the distribution of the absolute error of presented method and RBF
collocation method with Gaussian RBF, $N=20\times12$ with $c=0.3041$
for RBF collocation and $c=0.01$ for Presented method. The reported
results show that more accurate approximate solutions can be
obtained using more mesh points. Comparison of numerical results
show that the presented method is more accurate than the existing
RBF methods. Graphs of maximum absolute error versus shape parameter
with Gaussian RBF, $N=10\times6$ are given in Figure~\ref{fig_ex3_2}
which show that remain stable as shape parameter gets smaller in
contrast with the collocation radial basis functions methods.
\end{example}

%%%%%%%%%%%%%%%%%%%ex3_2
\begin{example}\label{ex3_2}
Consider the Poisson's equation \cite{schaback2},
\begin{equation*}\label{ex3_2:01}
\Delta u=2 e^{x-y},\hspace{.5 cm} (x,y)\in[0,1]\times[0,1]
\end{equation*}
with the nonhomogeneous Dirichlet boundary conditions
\begin{equation*}\label{ex3_2:02}
\begin{array}{lll}
u(0,y)&=g_{1}(y),\hspace{.5 cm}u(1,y)&=g_{2}(y),\\
u(x,0)&=h_{1}(x),\hspace{.5 cm}u(x,1)&=h_{2}(x).\\
\end{array}
\end{equation*}
The exact solution is given by
\begin{equation*}\label{ex3_2:03}
u(x,y)=e^{x-y}+e^{x} \cos y.
\end{equation*}
\begin{table}
\centering \scriptsize{\begin{tabular}{ |c|c|c|c|c| }
  \hline
 N&$5\times5$& $10\times10$  & $15\times15$ & $20\times 20$ \\
 \hline
 RBF collocation&1.56591e-4&3.89263e-11 &8.55909e-19&4.57185e-27\\
 \hline
 Presented method&8.12108e-9&4.6856e-15 &3.36241e-23&1.92864e-32\\
 \hline
\end{tabular}}
\caption{\scriptsize{Maximum absolute errors, comparison of results
for Example \ref{ex3_2}.}}\label{tab_ex3_2_1}
\end{table}
\begin{figure}
\centering
\includegraphics[scale=.6]{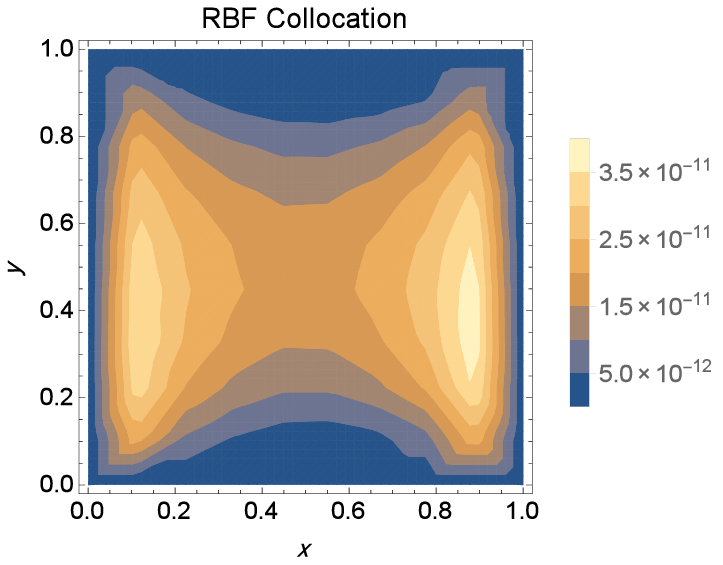}
\includegraphics[scale=.6]{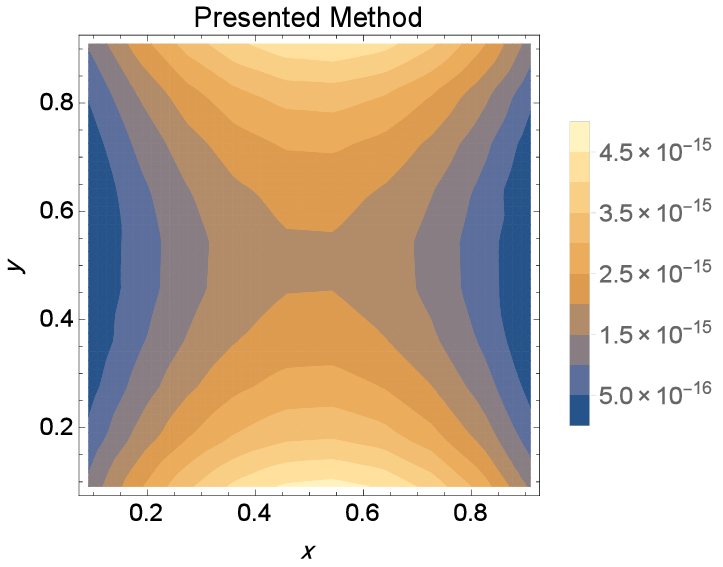}
\caption{\scriptsize{Distribution of the absolute error with
Gaussian RBF, $N=10\times10$ and $c=0.01$, for Example
\ref{ex3_2}.}}\label{fig_ex3_2_1}
\end{figure}
\begin{figure}
\centering
\includegraphics[scale=.8]{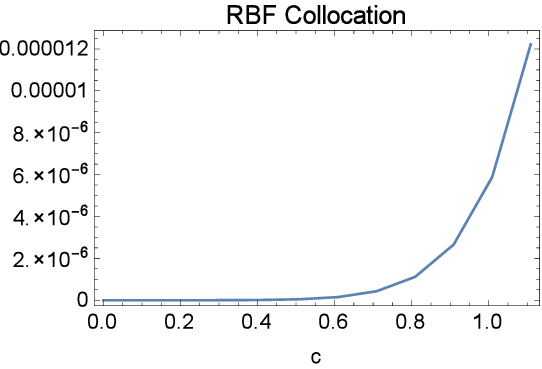}
\includegraphics[scale=.8]{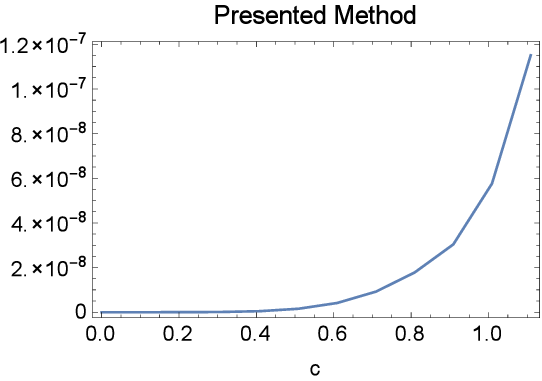}
\caption{\scriptsize{ Graphs of maximum absolute error versus shape
parameter with Gaussian RBF, $N=10\times10$, for Example
\ref{ex3_2}.}}\label{fig_ex3_2_2}
\end{figure}
The Maximum absolute errors are presented in Table~\ref{tab_ex3_2_1}
for various values of $N$. For comparison, the best result reported
in \cite{schaback2} has $1.28\times 10^{-4}$ maximum absolute error
with $81$ collocation points and $c=1.2$ shape parameter.
Figure~\ref{fig_ex3_2_1} shows the distribution of the absolute
error of presented method and RBF collocation method with Gaussian
RBF, $N=10\times10$ with $c=0.01$ for RBF collocation and Presented
method. The reported results show that more accurate approximate
solutions can be obtained using more mesh points. Comparison of
numerical results show that the presented method is more accurate
than the existing RBF methods. Graphs of maximum absolute error
versus shape parameter with Gaussian RBF, $N=10\times10$ are given
in Figure~\ref{fig_ex3_2_2}, which show that the presented method is
more accurate than RBF collocation method for various shape
parameters.
\end{example}

%%%%%%%%%%%%%%%%%%%ex4
\begin{example}\label{ex4}
Consider the Poisson's equation \cite{fassh_mixed},
\begin{equation*}\label{ex4:01}
\Delta u=\sin x-\sin^{3}x,\hspace{.5 cm}
(x,y)\in[0,\frac{\pi}{2}]\times[0,2],
\end{equation*}
with the Dirichlet and Neumann boundary conditions
\begin{equation*}\label{ex4:02}
u(0,y)=u_{x}(\frac{\pi}{2},y)=u_{y}(x,0)=u_{y}(x,2)=0.
\end{equation*}
The exact solution is given by
\begin{equation*}\label{ex4:03}
u(x,y)=-\frac{1}{4} \sin(x)-\frac{1}{36} \sin(3x).
\end{equation*}
\begin{table}
\centering \scriptsize{\begin{tabular}{ |c|c|c|c|c|c| }
  \hline
 N&$\rho_{K}$\cite{fassh_mixed}& $\rho_{H}$\cite{fassh_mixed} & $\rho_{RBFC}$ & $\rho_{1}$&$\rho_{2}$\\
 \hline
 $5\times5$&2.181029e-2&4.327029e-2&1.56966e-2& 1.2886e-3&2.31536e-2\\
 \hline
 $7\times7$&6.910084e-3&1.871798e-4&7.45327e-3& 1.34064e-5&1.33894e-3\\
 \hline
$10\times10$&9.265197e-5&5.126676e-5&5.75242e-4& 3.29045e-8&5.32917e-6\\
 \hline
 $14\times14$&1.138751e-5&1.725526e-6&5.59595e-5& 4.62586e-11&1.74509e-9\\
 \hline
 $20\times20$&5.501057e-6&6.217559e-7&1.34064e-6& 8.15272e-17&1.42493e-15\\
 \hline
\end{tabular}}
\caption{\scriptsize{Relative errors, comparison of results for
Example \ref{ex4} and $c=0.4641,0.01$ for  presented method and
$c=0.5641$ RBF collocation..}}\label{tab_ex4_1}
\end{table}
\begin{figure}
\centering
\includegraphics[scale=.6]{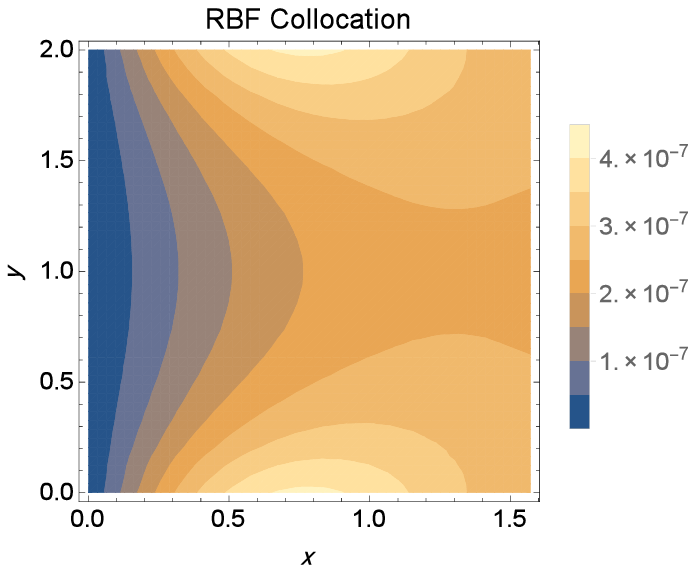}
\includegraphics[scale=.6]{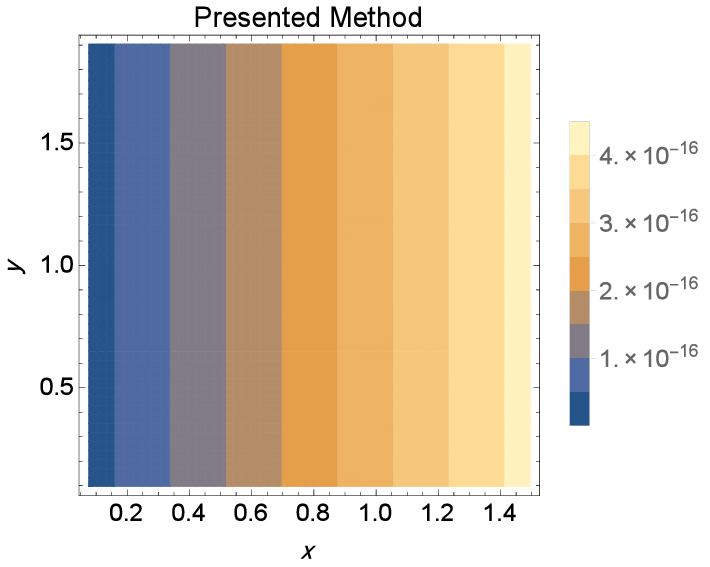}
\caption{\scriptsize{Distribution of the absolute error with
Gaussian RBF, $N=20\times20$ and $c=0.5641$ for RBF collocation and
$c=0.01$ for Presented method, for Example
\ref{ex4}.}}\label{fig_ex4_1}
\end{figure}
\begin{figure}
\centering
\includegraphics[scale=.8]{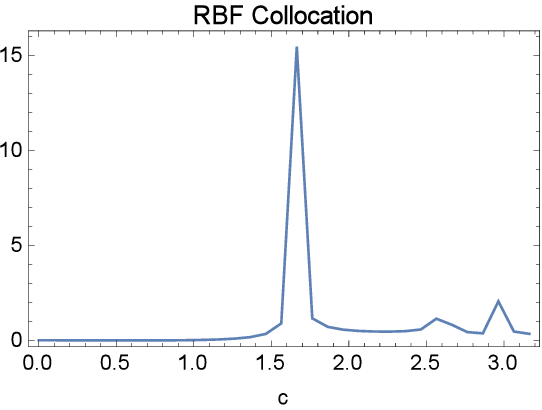}
\includegraphics[scale=.8]{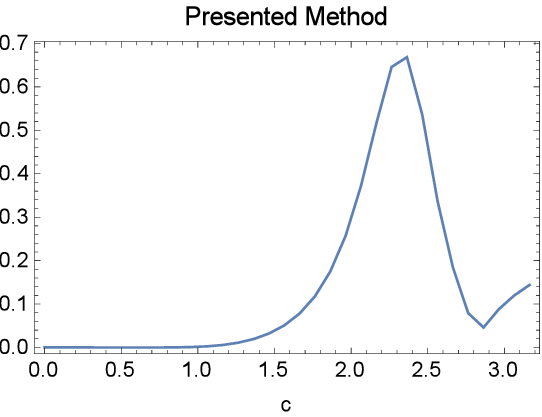}
\caption{\scriptsize{ Graphs of maximum absolute error versus shape
parameter with Gaussian RBF, $N=7\times7$, for Example
\ref{ex4}.}}\label{fig_ex4_2}
\end{figure}
For this example, the relative errors are presented in
Table~\ref{tab_ex4_1} for various values of $N$ and they are
compared with the best reported results in \cite{fassh_mixed}
contain Kansa's and Hermit based RBF method and RBF collocation
method. $\rho_{RBFC}$ is relative error of RBF collocation method
with Gaussian RBF with $c=0.5641$ and $\rho_{1}$ and $\rho_{2}$ are
relative errors of presented method with Gaussian RBF with
$c=0.4641$ and $c=0.01$, respectively. $\rho_{K}$ and $\rho_{H}$ are
reported relative errors in \cite{fassh_mixed} with optimal shape
parameters for Kansa's and Hermit based RBF method, respectively.
Figure~\ref{fig_ex4_1} shows the distribution of the absolute error
of presented method and RBF collocation method with Gaussian RBF,
$N=20\times20$ with $c=0.5641$ for RBF collocation and $c=0.01$ for
Presented method. The reported results show that more accurate
approximate solutions can be obtained using more mesh points.
Comparison of numerical results show that the presented method is
more accurate than the existing RBF methods. Graphs of maximum
absolute error versus shape parameter with Gaussian RBF,
$N=7\times7$ are given in Figure~\ref{fig_ex4_2}.
\end{example}
%%%%%%%%%%%%%%%%%%%ex5
\begin{example}\label{ex5}
Consider the nonlocal multi--point Poisson's equation \cite{volkov},
\begin{equation*}\label{ex5:01}
\Delta u=f,\hspace{.5 cm} (x,y)\in[0,1]\times[0,2],
\end{equation*}
with the multi--point boundary conditions
\begin{equation*}\label{ex5:02}
\begin{array}{llll}
u(0,y)&=u(1,y)=0,\\
u(x,2)&=g(x),\\
u(x,0)&=\frac{1}{4}u(x,\frac{3}{5})+\frac{1}{2}u(x,\frac{6}{5})+\frac{1}{4}u(x,\frac{9}{5}).\\
\end{array}
\end{equation*}
The functions $f$ and $g$ are given such that the exact solution is,
\begin{equation*}\label{ex5:03}
u(x,y)=\frac{1}{500}\left((e^{\pi x}-1)(e^{\pi
x}-e^{\pi})\sin(\frac{5\pi}{6}y)+e^{\pi
y(\frac{3}{5}-y)(\frac{6}{5}-y)(\frac{9}{5}-y)}\sin(\pi x)\right).
\end{equation*}
\begin{table}
\centering \scriptsize{\begin{tabular}{ |c|c|c|c|c| }
  \hline
 N&$5\times10$& $8\times16$  & $10\times20$ & $12\times 24$ \\
 \hline
 RBF collocation&4.09711e-2&1.70686e-3 &9.94226e-5& 4.15599e-6\\
 \hline
 Presented method&5.47254e-3&1.09398e-4 &1.44713e-5& 2.80392e-6\\
 \hline
\end{tabular}}
\caption{\scriptsize{Maximum absolute errors, comparison of results
for Example \ref{ex5}.}}\label{tab_ex5_1}
\end{table}
\begin{figure}
\centering
\includegraphics[scale=.6]{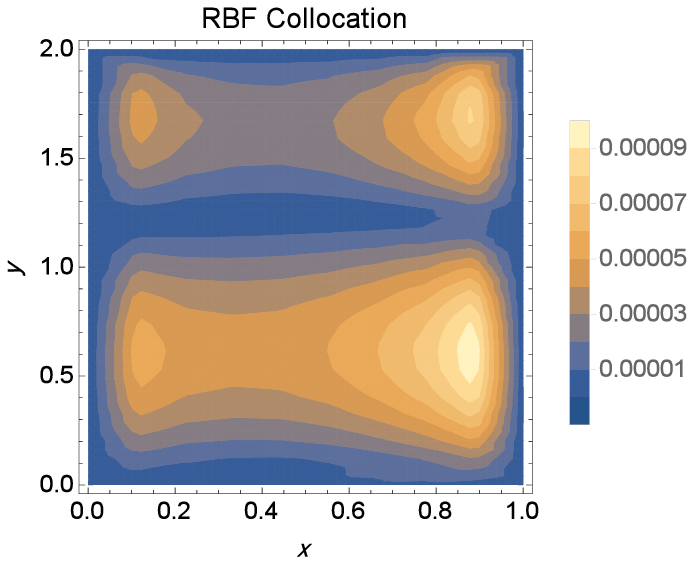}
\includegraphics[scale=.6]{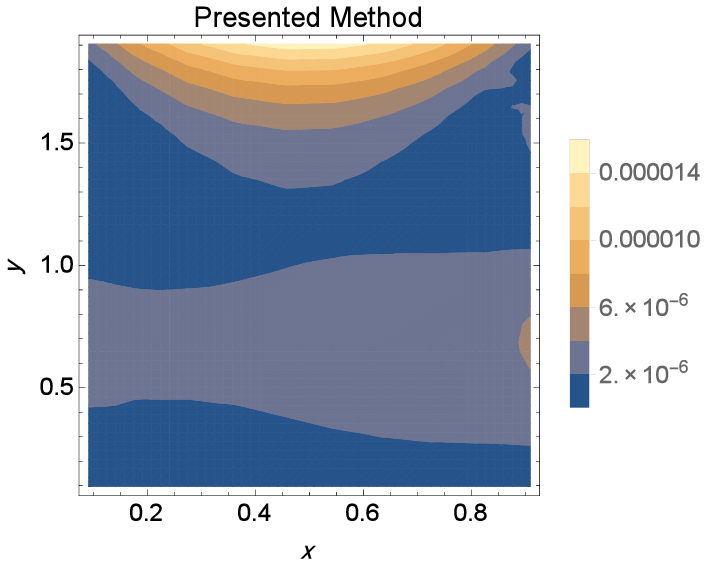}
\caption{\scriptsize{Distribution of the absolute error with
Gaussian RBF, $N=10\times20$ and $c=0.01$, for Example
\ref{ex5}.}}\label{fig_ex5_1}
\end{figure}
\begin{figure}
\centering
\includegraphics[scale=.8]{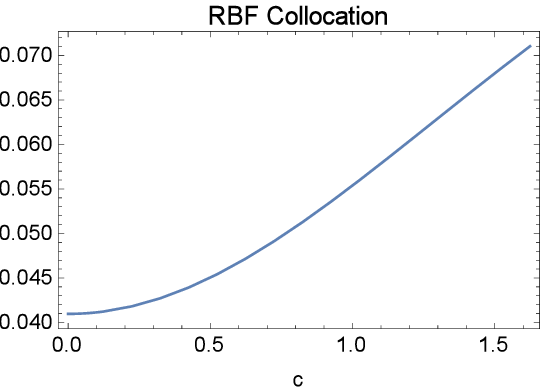}
\includegraphics[scale=.8]{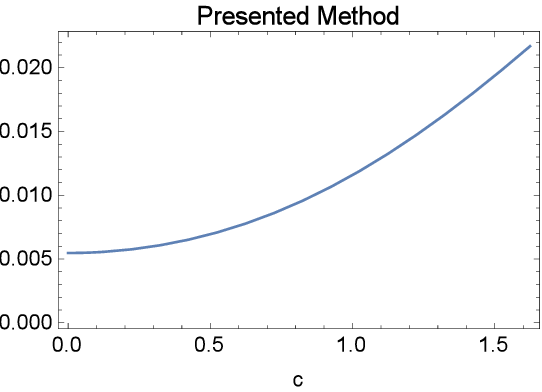}
\caption{\scriptsize{ Graphs of maximum absolute error versus shape
parameter with Gaussian RBF, $N=5\times10$, for Example
\ref{ex5}.}}\label{fig_ex5_2}
\end{figure}
The maximum absolute errors are presented in Table~\ref{tab_ex5_1}
for various values of $N$ and they are compared with the RBF
collocation method. The Gaussian RBF with $c=0.01$ is used for
presented method and RBF collocation method. Figure~\ref{fig_ex5_1}
shows the distribution of the absolute error of presented method and
RBF collocation method with Gaussian RBF, $N=10\times20$ and
$c=0.01$. The reported results show that more accurate approximate
solutions can be obtained using more mesh points. Comparison of
numerical results show that the presented method is more accurate
than the RBF collocation method. Graphs of maximum absolute error
versus shape parameter with Gaussian RBF, $N=5\times10$ are given in
Figure~\ref{fig_ex5_2}.
\end{example}
%%%%%%%%%%%%%%%%%%%ex6
\begin{example}\label{ex6}
Consider the Poisson's equation \cite{tsai},
\begin{equation*}\label{ex6:01}
\Delta u(x,y,z)=\frac{6}{(4+x+y+z)^{3}},\hspace{.5 cm} (x,y,z)\in
\Omega=[-\frac{1}{2},\frac{1}{2}]^{3},
\end{equation*}
with the Dirichlet boundary conditions
\begin{equation*}\label{ex6:02}
u=g(\textbf{x}), \hspace{.5 cm} \textbf{x} \in \Gamma,
\end{equation*}
where $\Gamma$ is the boundary of $\Omega$ and $g$ is given such
that the exact solution is,
\begin{equation*}\label{ex6:03}
u(x,y,z)=\frac{1}{(4+x+y+z)}.
\end{equation*}
\begin{table}
\centering \scriptsize{\begin{tabular}{ |c|c|c|c|c| }
  \hline
 N&$4\times4\times4$& $5\times5\times5$  & $6\times 6\times 6$ & $7\times 7\times 7$ \\
 \hline
 RBF collocation&3.8223e-5&4.86452e-6 &6.73616e-7& 8.47629e-8\\
 \hline
 Presented method&1.02919e-7&1.49101e-8 &2.4369e-9& 3.43708e-10\\
 \hline
\end{tabular}}
\caption{\scriptsize{Maximum absolute errors, comparison of results
for Example \ref{ex6} with $N=7\times7\times7$ and
$c=0.01$.}}\label{tab_ex6_1}
\end{table}
\begin{figure}
\centering
\includegraphics[scale=.8]{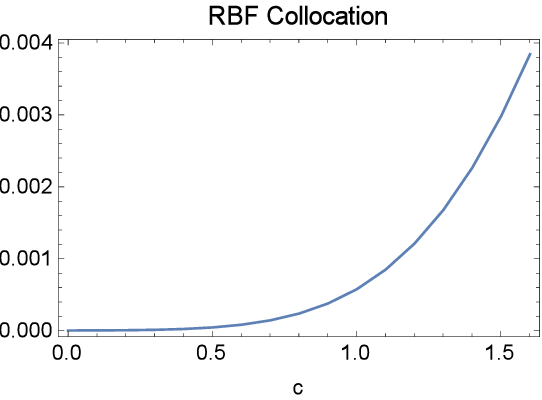}
\includegraphics[scale=.8]{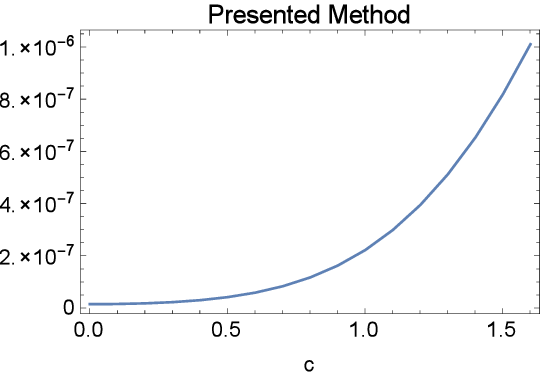}
\caption{\scriptsize{ Graphs of maximum absolute error versus shape
parameter with Gaussian RBF, $N=5\times5\times5$, for Example
\ref{ex6}.}}\label{fig_ex6_1}
\end{figure}
The maximum absolute errors are presented in Table~\ref{tab_ex6_1}
for various values of $N$ and they are compared with the RBF
collocation method. The Gaussian RBF with $c=0.01$ is used for
presented method and RBF collocation method. For comparison, the
best result reported in \cite{tsai} has $10^{-5}$ maximum absolute
error with $7\times7\times7$ points. The reported results show that
more accurate approximate solutions can be obtained using more mesh
points. Graphs of maximum absolute error versus shape parameter with
Gaussian RBF, $N=5\times5\times5$ are given in
Figure~\ref{fig_ex6_1}, which show that the presented method is more
accurate than RBF collocation method for various shape parameters.
\end{example}
%%%%%%%%%%%%%%%%%%%%
%%%%
%   %
%   %
%%%%
%   %
%   %
%%%%
%%%%%%%%%%%%%%%%%%%%

\section{Conclusions}
In this paper, we introduce a new approach for the imposing various
boundary conditions on radial basis functions and their application
in pseudospectral radial basis function method. The various boundary
conditions such as Dirichlet, Neumann, Robin, mixed and multi--point
boundary conditions for one, two and three-dimensional problems,
have been considered. Here we propose a new technique to force the
radial basis functions to satisfy the boundary conditions exactly.
Some new kernels are constructed using general kernels in a manner
which satisfies required conditions and we prove that if the
reference kernel is positive definite then the newly constructed
kernel is positive definite, also. Furthermore, we show that the
collocation matrix is nonsingular if some conditions are satisfied.
It can improve the applications of existing methods based on radial
basis functions especially the pseudospectral radial basis function
method to handling the differential equations with more complicated
boundary conditions. Several examples with various boundary
conditions have been considered for validation of the proposed
technique and the results are compared with the RBF collocation
method and the best-reported results in the literature.

% ----------------------------------------------------------------

\bibliographystyle{amsplain}

\end{document}